\documentclass[12pt,oneside,reqno]{amsart}

\usepackage{amsmath, amsthm, amsfonts, amsfonts,amssymb,amscd,amsmath,latexsym,enumerate,verbatim, calc, quotmark, mathtools, txfonts, calrsfs, imakeidx, fancyhdr, listings, lipsum, tkz-euclide}

\usepackage[T1]{fontenc}
\usepackage[utf8]{inputenc}

\usepackage[all, cmtip]{xy}

\usepackage[pagebackref]{hyperref}
\hypersetup{
     colorlinks = true, linkcolor = red,
     citecolor   = blue,
     urlcolor    = blue,
}

\def\NZQ{\mathbb}               
\def\NN{{\NZQ N}}

\def\ZZ{{\NZQ Z}}

\def\m{\mathfrak{m}}
\def\n{\mathfrak{n}}

\def\e{\varepsilon}

\def\GL{{\operatorname{GL}}}          
\def\M{{\operatorname{M}}}            
\def\vpmod{\!\!\pmod}

\DeclareMathOperator*{\depth}{depth}
\DeclareMathOperator*{\Tor}{Tor}
\DeclareMathOperator*{\ITor}{ITor}
\DeclareMathOperator*{\Ext}{Ext}
\DeclareMathOperator*{\Der}{Der}
\DeclareMathOperator*{\Ch}{char}

\DeclareMathOperator*{\ann}{ann}
\DeclareMathOperator*{\Ho}{H}
\DeclareMathOperator*{\IH}{IH}

\newtheorem{theorem}{Theorem}[section]
\newtheorem{lemma}[theorem]{Lemma}
\newtheorem{corollary}[theorem]{Corollary}
\newtheorem{proposition}[theorem]{Proposition}
\newtheorem{remark}[theorem]{Remark}

\newtheorem{definition}[theorem]{Definition}

\newtheorem{question}[theorem]{Question}

\newtheorem*{acknowledgement}{Acknowledgement}

\title{Ascent and descent of the Golod property along algebra retracts}
\author{{Anjan Gupta}\\
Department of Mathematics, IIT Bombay}

\thanks{Corresponding author: Anjan Gupta; 
{\it email: agmath@gmail.com, anjan@math.iitb.ac.in}}

\begin{document}
\begin{abstract} 
We study ascent and descent of the Golod property along an algebra retract. We characterise trivial extensions of modules, fibre products of rings to be Golod rings. We present a criterion for a graded module over a graded affine algebra of characteristic zero to be a Golod module.  
\end{abstract}
\maketitle



\section{Introduction}
Let  $R$ be a local ring with maximal ideal $\m$ and residue field $R/\m = k$. Let $M$ be a finitely generated $R$-module. The generating function of the sequence of Betti numbers of the minimal free resolution of $M$ over $R$ is a formal power series in $\ZZ[|t|]$. This series is called the Poincar\'e series of $M$ over $R$ and is denoted by $P^R_M(t)$ (Definition \ref{poin}). 
J-P. Serre showed that this series is coefficient-wise bounded above by a series representing a rational function. The module $M$ is said to be a Golod module when the Poincar\'e series coincides with the upper bound given by Serre (Definition \ref{p3}). 
The ring $R$ is said to be a Golod ring if its residue field $k$ is a Golod $R$-module. We refer the reader for details regarding Golod rings and Golod modules to the survey article \cite{av} by Avramov. The main objectives of this article are to study transfer of the Golod property along algebra retracts and more generally large homomorphisms, to establish a connection between the Golod property of a module and its trivial extension and finally to characterise the Golod property of fibre products of local rings.

%

A subring of a ring is called an algebra retract if the inclusion map has a left inverse (Definition \ref{p1}). Several authors have studied how ring-theoretic properties transfer along algebra retracts from different perspectives. Basic properties like normality of domains, semi normality, regularity, complete intersection, Koszul, Stanley-Reisner are known to descend along algebra retracts (see \cite{apassov}, \cite{epstein}, \cite{douglas}, \cite{ogushi}). On the other hand, properties like Cohen-Macaulay, Gorenstein are not inherited by an algebra retract in general. We refer the reader to \cite{epstein} for a very good exposition on this theme. In the present article we prove the following:

\begin{theorem}\label{i1}
Let $j : (R, \mathfrak{m}) \rightarrow (A, \mathfrak{n})$ be an algebra retract with a section $p: (A, \mathfrak{n}) \rightarrow (R, \mathfrak{m})$. Let $M$ be a finitely generated $R$-module which is Golod when viewed as an $A$-module via the homomorphism $p$. Then $M$ is also a Golod $R$-module.
\end{theorem}

The Golod property does not ascend along an algebra retract in general as seen by any non-Golod local ring containing its residue field. So certain assumptions are necessary for an affirmative answer. Our main result stated below presents one such assumption.

\begin{theorem}\label{i2}
Let $j:(R, \m) \rightarrow (A, \n)$ be an algebra retract which admits sections (possibly equal) $p$ and $p'$. Let $\ker(p) = I$ and $\ker(p') = I'$ satisfy $II' = 0$. 
Consider $R$-module structures on $I$, $I'$ via the retract map $j$.
Then the ideal $I$ is a Golod $R$-module if and only if $I'$ is so. 

Let $N$ be an $R$-module. If we consider $N$ as an $A$-module via any of the maps $p, p'$, then $N$ is a Golod $A$-module if and only if $N$ is a Golod $R$-module and $I$ (equivalently $I'$) is a Golod $R$-module. In particular, $A$ is a Golod ring if and only if $I$ (equivalently $I'$) is a Golod $R$-module.
\end{theorem}

As an application we prove the following theorem.
\begin{theorem}\label{i3}
Let $(R, \m)$ be a local ring and $M$ be an $R$-module. Let $A = R \lJoin M$ be the trivial extension of $R$ by $M$. Then $A$ is a Golod ring if and only if $M$ is a Golod $R$-module. 
\end{theorem}

The above result gives us an efficient method to study the Golod property of modules with results available to characterise Golod rings. We demonstrate this by giving a new characterisation of regularity of a ring in terms of the Golod property of its canonical module (Corollary \ref{ap4}). 

Next we turn our attention to study the Golod property of  fibre products of local rings. The fibre product $R_1 \times_S R_2$ of two local epimorphisms $\e_i : (R_i, \m_i, k) \twoheadrightarrow (S, \n, k), i = 1, 2$ of Noetherian local rings is also a Noetherian \cite[Theorem 2.1]{ogoma} local ring and shares the same residue field $k$. We  refer the reader to Definition \ref{fibre} for an explicit construction. 
In the recent article
\cite[Theorem 1.8]{ananth} the authors present criteria for the fibre product $R_1 \times_S R_2$ to be Cohen Macaulay, Gorenstein. Dress and Kr\"amer computed the Poincar\'e series of the residue field $k$ over  $R_1 \times_k R_2$ in \cite[Satz 1]{dress}. The result was reproved later in \cite[\S1]{moore} by presenting an explicit construction of the minimal free resolution of $k$ over $R_1 \times_k R_2$. The series was also computed for fibre products over arbitrary rings in \cite{her1}, \cite{les2} under some hypothesis. In \cite[Theorem 4.1]{les1}, Lescot showed that the fibre product $R_1 \times_k R_2$ is Golod if and only if both $R_1$ and $R_2$ are so. His result in particular shows that for a local ring $(R, \m, k)$, the fibre product $R \times_k R$ is Golod whenever $\m$ is a Golod $R$-module. We generalize this fact as follows as an application of Theorem \ref{i2}.

\begin{theorem}
Let $I$ be an ideal of a local ring $R$. Let for $n \geq 2$

\[A_n = \underbrace{R \times_{R/I} R \ldots \times_{R/I} R}_{n\ \text{times}}.\]
Then the following are equivalent.
\begin{enumerate}
\item $I$ is a Golod $R$-module.

\item The fibre product $A_n$ is a Golod ring for all $n \geq 2$.

\item The fibre product $A_n$ is a Golod ring for some $n \geq 2$.
\end{enumerate}
\end{theorem}

The notion of large homomorphisms was introduced by Levin in \cite{lev4}. A ring homomorphism is large if it induces a surjection on Tor algebras (Definition \ref{ap5}).  For example if $R \hookrightarrow A \twoheadrightarrow R$ is an algebra retract, then the surjection $A \twoheadrightarrow R$ is a large homomorphism. Our final result in this article strengthens Theorem \ref{i1} as follows:

\begin{theorem}
Let $f : (R, \m) \twoheadrightarrow (S, \n)$ be a large homomorphism of local rings and $M$ be an $S$-module. Then $M$ is a Golod $S$-module if $M$ is so as an $R$-module.
\end{theorem}

Now we briefly describe the layout of the article. In section 2 we collect some definitions, notations, and known results which we will use. In section 3 we characterise Golod modules in terms of vanishing of Massey products on Koszul homologies. As an offshoot we prove a criterion for a graded module over a graded affine algebra of characteristic zero to be a Golod module  generalising a result of Herzog and Huneke (Theorem \ref{huneke}).
In section 4 we discuss how the Golod property transfers along algebra retracts. We prove Theorems \ref{i1}, \ref{i2} in this section. In the final section we prove rest of the results as an application of Theorem \ref{i2}.

{\bf
All rings in this article are Noetherian local rings with $1 \not= 0$. All modules are nonzero and finitely generated. All homomorphisms of local rings are local, i.e. they send non-units to non-units. Throughout this article, the notation $(R, \m, k)$ denotes a local ring $R$ with maximal ideal $\m$ and residue field $k$.}

\section{Preliminaries}
We begin by reminding the reader of a few elementary definitions without specific reference. Most terms and concepts involved can be found in \cite{av}. 
Let $(R, \m, k)$ be a local ring.

\begin{definition}{\rm(Poincar\'e Series)}\label{poin}
The Poincar\'e series of an $R$-module $M$ is defined as $$P^R_M(t) = \sum_{i \geq 0} \dim_k {\Tor}^R_i(k, M)t^i \in \ZZ[|t|].$$
\end{definition}

We denote Koszul complexes of $R$ and $M$ on a minimal set of generators of $\m$ by $K^R$ and $K^M$ respectively. For notational simplicity we abbreviate  $\Ho_i(K^R)$ and $\Ho_i(K^M)$ as $\Ho_i(R)$ and $\Ho_i(M)$ respectively.

\begin{definition}{\rm(Koszul polynomials)}
The Koszul polynomial of an $R$-module $M$ is defined as $\kappa^R_M(t) = \sum_{i \geq 0} \dim_k \Ho_i(M)t^i$. When there is no chance of confusion, we denote the Koszul polynomial simply by $\kappa_M(t)$ dropping the superscript $R$. Likewise the Koszul polynomial of $R$ is $\kappa_R(t) = \sum_{i \geq 0} \dim_k \Ho_i(R)t^i$.
\end{definition}

The following result in \cite[Corollary 2]{her2} gives an explicit basis of Koszul homologies.

\begin{theorem}\label{pher}
Let $S = k[X_1, \ldots, X_n]$, $\deg(X_i) = a_i > 0$ be a graded polynomial ring over a field $k$ of characteristic zero. Let $I$ be a homogeneous ideal different from $S$. Suppose a graded minimal free resolution of $S/I$ over $S$ is given by
\[0 \rightarrow F_p \xrightarrow{\phi_p} F_{p-1} \xrightarrow{\phi_{p-1}} \ldots \rightarrow F_1 \xrightarrow{\phi_1} F_0 \rightarrow S/I \rightarrow 0.\] 

Let $F_i$ be a free module on a homogeneous set of basis $\{f_{i1}, f_{i2}, \ldots, f_{i b_i}\}$ and $\phi_i(f_{ij}) = \sum_{k = 1}^{b_{i -1}}\alpha^{(i)}_{j k} f_{i - 1 k}$. Then for all $l = 1, 2, \ldots, p$, the elements

\[\sum_{1 \leq i_1 < i_2 < \ldots < i_l \leq n} a_{i_1}a_{i_2} \ldots a_{i_l} \sum_{j_2 = 1}^{b_{l-1}} \ldots \sum_{j_l = 1}^{b_1}  c_{{j_1}, \ldots, j_l} \overline{\frac{\partial( \alpha^{(l)}_{j_1, j_2}, \alpha^{(l - 1)}_{j_2, j_3}, \ldots, \alpha^{(1)}_{j_l,1})}{\partial(X_{i_1}, X_{i_2}, \ldots, X_{i_l})}}e_{i_1} \wedge e_{i_2} \wedge \ldots \wedge e_{i_l}; j_1 = 1, \ldots, b_l \]
are cycles of $K^R_l$ whose homology classes form a $k$-basis of $\Ho_l(R)$. Here $ c_{{j_1}, \ldots, j_l}$ are rational numbers determined by the degrees of the $\alpha_{j k}^{(i)}$ and the elements $e_{i_1} \wedge e_{i_2} \wedge \ldots \wedge e_{i_l}$ form an $R$-basis of the free module $K^R_l = \wedge ^l(\oplus_{i = 1}^n Re_l)$. The overline denote the class of the Jacobian modulo $I$.
\end{theorem}

\begin{definition}\label{p3}{\rm \cite{lev1} (Golod homomorphisms, Golod modules and Golod rings)}

Let $\phi: (R, \m, k) \rightarrow (S, \n, k)$ be a homomorphism of local rings. Given an $S$-module $M$, it carries the usual structure of an R-module via $\phi$. Let $\ITor^R(S, k)$ be the kernel of the augmentation map $\Tor^R(S, k) \twoheadrightarrow k$.
From the standard change of rings spectral sequence 
$\Tor^S_p(M, \Tor^R_q(S, k)) \underset{p}{\Rightarrow} \Tor^R_{p + q}(M, k)$ 
for an $S$-module $M$, one has a term-wise inequality of Poinca\'re series 
\[P^S_M(t) \leq \frac{P^R_M(t)}{1 - t(P^R_S(t) - 1)}.\]
If equality holds and $\n \ITor^R(S, k) = 0$, then $M$ is called a $\phi$-Golod module and if the residue field $k$ is $\phi$-Golod, then one says that $\phi$ is a Golod homomorphism.

Now assume that $M$ is an $R$-module. For a local ring $(R,m,k)$, there is a regular local ring $(Q, \mathfrak{l}, k)$ and a surjection $\phi : Q  \twoheadrightarrow \hat{R}$ such that $edim(R) = edim(Q)$. In this context the above inequality becomes

\[P^R_M(t) = P^{\hat{R}}_{\hat{M}}(t) \leq \frac{P^Q_{\hat{M}}(t)}{1 - t(P^Q_{\hat{R}}(t) - 1)} = \frac{\kappa_M(t)}{1 - t(\kappa_R(t) - 1)}\]

If equality holds, then $M$ is called a Golod $R$-module. The ring $R$ is said to be a Golod ring if the residue field $k$ is a Golod $R$-module. 
\end{definition}

The following result is implicit in \cite[Theorem 1.1]{lev1}. The proof follows verbatim as that of \cite[Theorem 1.3, (ii)]{lev3}.
\begin{theorem}\label{p4}
Let $\phi: (R, \m, k) \twoheadrightarrow (S, \n, k)$ be a surjective homomorphism of local rings and $M$ be an $S$-module. Then $M$ is a $\phi$-Golod module if and only if in the change of ring spectral sequence $\Tor^S_p(M, \Tor^R_q(S, k)) \underset{p}{\Rightarrow} \Tor^R_{p + q}(M, k)$, the following hold.
\begin{enumerate}
\item
$d^r_{pq} = 0$ for all $r \geq 2, q > 0$
\item
$E^{\infty}_{pq} = 0$ for all $q >0$
\end{enumerate}
\end{theorem}

\begin{remark}\label{p4s}
Let $(R, \m)$ be a local ring and $\phi : Q \rightarrow \hat{R}$ be a minimal Cohen presentation of $R$. Let $M$ be an $R$-module.  We have $ \Tor^Q_i(\hat{R}, k) = \Ho_i(R)$ and   $\Tor^Q_i(\hat{M}, k) = \Ho_i(M)$. The module $\hat{M}$ is a $\phi$-Golod module if and only $M$ is a Golod $R$-module. The change of ring spectral sequence reduces to the spectral sequence $\Tor^R_p(M, \Ho_q(R)) \underset{p}{\Rightarrow} \Ho_{p + q}(M)$. Therefore by Theorem \ref{p4}, the module $M$ is a Golod $R$-module if and only the spectral sequence  $\Tor^R_p(M, \Ho_q(R)) \underset{p}{\Rightarrow} \Ho_{p + q}(M)$ satisfies conditions 1 and 2 of Theorem \ref{p4}.
\end{remark}

The following theorem of Levin implies that if a ring admits a Golod module, then the ring must be a Golod ring.

\begin{theorem}{\rm \cite[Theorem 1.1]{lev1}}\label{levingm}
Let $\phi: (R, \m, k) \rightarrow (S, \n, k)$ be a local homomorphism and $M$ an $S$-module. Then the following are equivalent. 
\begin{enumerate}
\item
The $S$-module $M$ is $\phi$-Golod.

\item
The homomorphism $\phi$ is Golod  and the induced map $\phi_* : \Tor^R(M, k) \rightarrow \Tor^S(M, k)$ is injective.
\end{enumerate}
\end{theorem}

\begin{definition}{\rm (Algebra retracts)}\label{p1}
A ring (DG algebra) homorphism $j : R \rightarrow A$ is called an algebra (DG algebra) retract if there is a ring (DG algebra) homomorphism $p: A \rightarrow R$ such that $p \circ j = id$. We call $p$ a section of $j$. We also say that R is an algebra (DG algebra) retract of A.
\end{definition}

The following result of Herzog gives a relation between the Poincar\'e Series of a ring and that of its algebra retract.

\begin{theorem}{\rm \cite[Theorem 1]{her1}}\label{p5}
Let $j : (R, \mathfrak{m}) \rightarrow (A, \mathfrak{n})$ be an algebra retract with a section $p: (A, \mathfrak{n}) \rightarrow (R, \mathfrak{m})$. Let $M$ be a finitely generated $R$-module. Let $F_{*} \rightarrow R$ be the minimal resolution of $R$ by free $A$-modules and $G_{*} \rightarrow M$ be that of $M$ by free $R$-modules. Then $F_{*} \otimes_R G_{*}  \rightarrow M$ is the minimal resolution of $M$ by free $A$-modules and therefore $P^A_M(t) = P^A_R(t) P^R_M(t)$.
\end{theorem}

The following notion was introduced in \cite{lev4}.

\begin{definition}{\rm(Large homomorphisms)}\label{ap5}
Let $f : (R, \m, k) \twoheadrightarrow (S, \n, k)$ be a surjective local homomorphism. We say that $f$ is a large homomorphism if the induced map $f_* :  \Tor^R(k, k) \rightarrow \Tor^S(k, k)$ is a surjection equivalently $f^* :  \Ext_S(k, k) \rightarrow \Ext_R(k, k)$ is an injection.
\end{definition}

Note that the section map $p$ of an algebra retract is a large homomorphism. The following result is proved in \cite[Theorem 1.1]{lev4}.

\begin{theorem}\label{large}
Let $f : (R, \m, k) \twoheadrightarrow (S, \n, k)$ be a surjective local homomorphism of local rings. Then the following are equivalent.
\begin{enumerate}
 \item
The homomorphism $f$ is large.

\item
For any finitely generated $S$-module $M$, regarded as an $R$-module via $f$, the induced homomorphism $f_* :  \Tor^R(M, k) \rightarrow \Tor^S(M, k)$ is surjective. 

\end{enumerate}
\end{theorem}
We recall the definition of connected augmented DG algebra from \cite[Chapter 1]{gul}. The Massey products were defined more generally in \cite{may}. The following is an excerpt suitably tailored to our needs.
\begin{definition}{\rm(Trivial Massey products)}\label{p6.1}
Let $(R, \m, k)$ be a local ring and $U$ be a connected augmented DG algebra over $R$ with augmentation map $U \twoheadrightarrow k$. Let $P$ be a left DG module and $N$ be a right DG module over $U$. Let $\IH(U) = \ker \{\Ho(U) \rightarrow k\}$. For a homogeneous element $a \in U, N$ or $P$ we set $\bar{a} = (-1)^{\deg(a) + 1}a$. Assume that $v_1 \in \Ho(N)$, $v_2, \ldots, v_{n - 1} \in \IH(U)$ and $v_n \in \Ho(P)$ are homogeneous elements. We say that the Massey product $\langle v_1, \ldots, v_n \rangle = 0$ (trivial) if  for any set of elements $a_{ij}, 0 \leq i < j \leq n$, $(i,\ j) \not= (0, \ n)$ such that 
\begin{enumerate}
\item
$a_{i j} \in U$ for $ 1 \leq i < j \leq n - 1$,

\item $a_{0 j} \in N$ for $j = 1, 2, \ldots, n - 1$,

\item
$a_{i n} \in P$ for $i = 1, 2, \ldots, n - 1$,

\item
Class of $a_{i -1  i}$ is $v_i$,

\item
$d(a_{ij}) = \sum_{k = i + 1}^{j - 1}\overline{a_{i k}}a_{k j}$,
\end{enumerate}
there is an element $a_{0 n} \in N \otimes_U P$ satisfying $d(a_{0  n}) = \sum_{k = 1}^{n - 1}\overline{a_{0  k}}a_{k  n}$. The set $\{a_{ij} | 0 \leq i < j \leq n, (i,\ j) \not= (0, \ n) \}$ satisfying above five conditions is called the defining system for $v_1, \ldots, v_n$. Note that the matrix $A = (a_{i j})$ satisfies $d(A) = \bar{A}A$. 

If $\langle v_1, v_2, \ldots, v_n \rangle = 0$ for all possible choices of $v_1, v_2, \ldots, v_{n - 1} \in \IH(U)$, $v_n \in \Ho(P)$, 
then we say that $P$ admits trivial Massey products on homologies as left DG $U$-module. Similarly if $\langle v_1, v_2, \ldots, v_n \rangle = 0$ for all possible choices of $v_1 \in \Ho(N)$, $v_2, \ldots, v_n \in \IH(U)$, then we say that $N$ admits trivial Massey products on homologies as right DG $U$-module.
\end{definition}

The following lemma characterises Golod rings in terms of vanishing of Massey products on Koszul homologies. We refer the reader to \cite[Theorem 5.2.2]{av}, \cite[Theorems 1.5]{lev3} for a proof. 

\begin{lemma}\label{p6.2}
Let $(R, \m, k)$ be a local ring. Then the following are equivalent.

\begin{enumerate}
\item $R$ is a Golod ring.

\item
All the Massey products $\langle v_1, \ldots, v_n \rangle = 0$ for $v_i \in \Ho_{\geq 1}(R)$.

\item
There is a graded $k$-basis $\mathfrak{b}_R = \{h_{\lambda}\}_{\lambda \in \Lambda}$ of $\Ho_{\geq 1}(R)$ and a function $\mu : \sqcup_{i = 1}^{\infty} \mathfrak{b}_R^i \rightarrow  K^R$ such that $\mu(h_\lambda) \in Z(K^R)$ with  $cls(\mu(h_\lambda))= h_\lambda$ and 
\[ \partial \mu(h_{\lambda_1}, \ldots ,h_{\lambda_p}) = \sum_{j = 1}^{p-1}\overline{\mu(h_{\lambda_1}, \ldots ,h_{\lambda_j})}\mu(h_{\lambda_{j + 1}}, \ldots ,h_{\lambda_p}).\]
\end{enumerate}
\end{lemma}

We recall the definition of graded Lie algebras and other necessary terminology 
from \cite[\S 10]{av}. The following result is proved in \cite[Proposition A.1.10]{lem}. Another proof in the case of characteristic zero can be found in the corollary following \cite[Proposition 21.4]{felix}. In fact the same proof will go through in positive characteristic too after modifications.

\begin{proposition}\label{p8}
Let $V = \bigoplus_{i \geq 1} V_i$ be a positively graded vector space over a field $k$ and $L$ be a graded free Lie algebra on $V$. Then, any graded sub Lie algebra of $L$ is free.
\end{proposition}

Let $F \rightarrow k$ be an augmented DG $\Gamma$-algebra over a field $k$ with locally Noetherian homology. Then it is known that $\Tor^F(k, k)$ is a graded Hopf $\Gamma$-algebra which is locally finite dimensional over $k$ and there is a unique graded Lie algebra $L$ whose universal enveloping algebra is the graded vector space dual $\Tor^F(k, k)^{\lor}$. We call this Lie algebra $L$ the homotopy Lie algebra of $F$ and denote by $\pi(F)$. We refer the reader to \cite[\S 1]{av2} for more details. If $R\langle X \rangle$ is an acyclic closure of the residue field $k$ over a local ring $R$, then $\pi(R) = \Ho \Der^{\gamma}_R(R\langle X \rangle, R\langle X \rangle)$ (see \cite[Theorem 10.2.1]{av}).

We conclude this section stating the following result which follows from \cite[Proposition 3.1]{av2},  \cite[Theorem 3.4]{av1}.

\begin{theorem}\label{p9}
A local ring $R$ is Golod if and only if $\pi(K^R) = \pi^{\geq 2}(R)$ is a free Lie algebra.
\end{theorem}

%
\section{Characterisation of Golod modules}
We give a characterisation of Golod modules in terms of vanishing of Massey products on Koszul homologies. 
\begin{proposition}\label{p6.3}
Let $(R, \m, k)$ be a local ring and $M$ be an $R$-module. Then the following are equivalent.

\begin{enumerate}
\item[(i)]
$M$ is a Golod $R$-module.

\item[(ii)]
The ring $R$ is Golod and all the Massey products $\langle v_1, \ldots, v_n \rangle = 0$ for $v_1 \in \Ho(M)$ and $v_i \in \Ho_{\geq 1}(R)$ for $i = 2, \ldots, n$.

\item[(iii)]
There is a graded $k$-basis $\mathfrak{b}_R = \{h_{\lambda}\}_{\lambda \in \Lambda}$ of $\Ho_{\geq 1}(R)$, a graded $k$-basis $\mathfrak{b}_M = \{m_{\lambda}\}_{\lambda \in \Lambda'}$ of $\Ho(M)$ and a function 
\[\mu : (\sqcup_{i = 0}^{\infty} \mathfrak{b}_M \times \mathfrak{b}_R^i) \sqcup (\sqcup_{i = 1}^{\infty} \mathfrak{b}_R^i) \rightarrow K^M \sqcup K^R\]
such that $\mu(\sqcup_{i = 0}^{\infty} \mathfrak{b}_M \times \mathfrak{b}_R^i) \subset K^M$, $\mu(\sqcup_{i = 1}^{\infty} \mathfrak{b}_R) \subset K^R$, $\mu(h_\lambda) \in Z(K^R)$ with  $cls(\mu(h_\lambda))= h_\lambda$, $\mu(m_\lambda) \in Z(K^M)$ with  $cls(\mu(m_\lambda))= m_\lambda$ and 

\[ \partial \mu(h_{\lambda_1}, \ldots ,h_{\lambda_p}) = \sum_{j = 1}^{p-1}\overline{\mu(h_{\lambda_1}, \ldots ,h_{\lambda_j})}\mu(h_{\lambda_{j + 1}}, \ldots ,h_{\lambda_p}),\]

\[ \partial \mu(m_{\lambda_1}, h_{\lambda_2}, \ldots ,h_{\lambda_p}) = \sum_{j = 1}^{p-1}\overline{\mu(m_{\lambda_1},  h_{\lambda_2}, \ldots ,h_{\lambda_j})}\mu(h_{\lambda_{j + 1}}, \ldots ,h_{\lambda_p}).\]
\end{enumerate}
\end{proposition}

\begin{proof}    
If $M$ is Golod, then the ring $R$ is Golod and the map $\Ho(M) \rightarrow \Tor^R(M, k)$ is an injection by Theorem \ref{levingm}. The Massey products vanish by an argument implicit in \cite[Theorem 1.1]{lev1}(also see \cite[Theorem 3.8]{lev3}). So $(i)$ implies $(ii)$.

Suppose $(ii)$ holds. We choose a $k$-basis $\mathfrak{b}_M$ of $\Ho(M)$ and a $k$-basis $\mathfrak{b}_R$ of $\Ho_{\geq 1}(R)$ which admits a trivial Massey operation $\mu : \sqcup_{i = 1}^{\infty} \mathfrak{b}_R^i \rightarrow  K^R$ by Lemma \ref{p6.2}. The map $\mu$ can easily be extended to  $(\sqcup_{i = 0}^{\infty} \mathfrak{b}_M \times \mathfrak{b}_R^i) \sqcup (\sqcup_{i = 1}^{\infty} \mathfrak{b}_R^i)$ with required properties. So $(ii)$ implies $(iii)$.

The only non trivial part is that $(iii)$ implies $(i)$. We recall the proof of \cite[Theorem 5.2.2]{av}. We use the same notations as in \cite[Theorem 5.2.2]{av}. A minimal free resolution $G \twoheadrightarrow k$ is provided where $G_n = \bigoplus_{h + i_1 + i_2 + \ldots + i_p = n} K^R_h \otimes V_{i_1} \otimes_R \ldots \otimes_R V_{i_p}$ and each $V_n$ is a free $R$-module with basis $\{v_{\lambda} : n = |v_{\lambda}| = |{h_{\lambda}}| + 1\}_{\lambda \in \Lambda}$. The differential $\partial$ on $G$ is defined by 
\[\partial(a \otimes v_{\lambda_1} \otimes \ldots \otimes v_{\lambda_p}) =  \partial(a) \otimes v_{\lambda_1} \otimes \ldots \otimes v_{\lambda_p} + 
(-1)^{|a|}\sum_{j = 1}^pa\mu(h_{\lambda_1}, \ldots, h_{\lambda_j}) \otimes v_{\lambda_{j + 1}} \otimes \ldots \otimes v_{\lambda_p}. 
\]

We have a short split exact sequence $0 \rightarrow K^R \rightarrow G \rightarrow G \otimes_R V \rightarrow 0$. This gives the following long exact sequence of homologies 
\begin{align*}
&\cdots \bigoplus_{j \geq 1} {\Ho}_{n - j}(M \otimes_RG) \otimes_R V_{j + 1}  \xrightarrow{\delta} {\Ho}_{n}(M \otimes_R K^R) \rightarrow {\Ho}_{n}(M \otimes_R G)\\ &\rightarrow \bigoplus_{j \geq 1} {\Ho}_{n - 1 - j}(M \otimes_RG) \otimes_R V_{j + 1} \xrightarrow{\delta}  {\Ho}_{n - 1}(M \otimes_R K^R) \cdots.
\end{align*}
 Now the boundary map $\delta = 0$ if and only if the above long exact sequence splits into short exact sequences $0 \rightarrow \Ho_{n}(K^M) \rightarrow \Ho_{n}(M \otimes_R G) \rightarrow \bigoplus_{ j \geq 1} \Ho_{n - 1 - j}(M \otimes_RG) \otimes_R V_{j + 1}  \rightarrow 0$.
The later is equivalent to $\kappa_M(t) + P^R_M(t) \times t(\kappa_R(t) - 1) = P^R_M(t)$ i.e. $P^R_M(t) = \frac{\kappa_M(t)}{1 - t(\kappa_R(t) - 1)}$. So $M$ is a Golod module if and only if $\delta = 0$.

We consider the element $c(\lambda_1, \ldots, \lambda_n) = \sum_{j = 1}^{n - 1} {\mu(m_{\lambda_1},  h_{\lambda_2}, \ldots ,h_{\lambda_j})} v_{\lambda_{j + 1}}\otimes \ldots \otimes v_{\lambda_{n}} \in M \otimes_R G \otimes _R V$.  It is easy to see that  $c(\{\lambda_i\})$ is a cycle in $M \otimes_R G \otimes _R V$, i.e. $\partial c(\{\lambda_i\}) = 0$. The computation is similar to that involved to show that $\partial(\partial(v_{\lambda_1} \otimes \ldots \otimes v_{\lambda_n})) = 0$. The cycle $c(\{\lambda_i\})$  in $M \otimes_R G \otimes _R V$ lifts to a cycle $c'(\{\lambda_i\}) = \sum_{j = 1}^{n} \mu(m_{\lambda_1},  h_{\lambda_2}, \ldots ,h_{\lambda_j}) v_{\lambda_{j + 1}}\otimes \ldots \otimes v_{\lambda_{n}} \in M \otimes_R G$. So $\delta([c(\{\lambda_i)\}]) = 0$. We call $c(\{\lambda_i\})$ a special cycle of $M \otimes_R G \otimes _R V$. Therefore, to prove that $M$ is Golod (equivalently $\delta = 0$), it is enough to show that any cycle of $M \otimes_R G \otimes _R V$ is a finite $R$-linear sum of special cycles modulo boundary. Let $\mathfrak{S}$ be the $R$-submodule of $M \otimes_R G \otimes _R V$ generated by special cycles. We show that $Z(M \otimes_R G \otimes_R V)  \subset  \mathfrak{S} + B(M \otimes_R G \otimes_R V)$.

Let $T_j(V)$ denote the $j$-th graded component of the tensor $R$-algebra $T(V)$ on $V$. Set $\mathfrak{F}_n = \bigoplus_ {j = 1}^{n}K^M \otimes T_j(V)$. Then $\{\mathfrak{F}_n\}_{n \geq 1}$ is a filtration of $ M \otimes_R G \otimes _R V$ by sub-complexes.
Any cycle $x \in M \otimes_R G \otimes _R V$  is in $\mathfrak{F}_n$ for some $n$. So $x$ can be written as $x = \sum_{j = 1}^m z_j \otimes v_{\lambda_{1}^j} \otimes \ldots \otimes v_{\lambda_{n}^j} + y$, $z_j \in K^M$, $y \in \mathfrak{F}_{n - 1}$. We have $d(x) = \sum_{j = 1}^m d(z_j) \otimes v_{\lambda_{1}^j} \otimes \ldots \otimes v_{\lambda_{n}^j} + y' = 0$, $y' \in \mathfrak{F}_{n - 1}$. So each $z_j$ is a cycle of $K^M$. 

Pick $z_j \otimes v_{\lambda_{1}^j}\otimes \ldots \otimes v_{\lambda_{n}^j}$. Since $\{m_{\lambda}\}_{\lambda \in \Lambda'}$ is a basis of  of $\Ho(M)$, we can write $z_j = \sum _{i = 1}^{p_j} a^j_i \mu(m_{\delta^j_i}) + \partial(b_j)$, $a^j_i \in R$, $ b_j \in K^M$. It is easy to see that $z_j \otimes v_{\lambda_{1}^j}\otimes \ldots \otimes v_{\lambda_{n}^j} - \sum _{i = 1}^{p_j} a^j_i c(\delta^j_i, \lambda^j_1, \ldots, \lambda^j_n) - \partial(b_j \otimes v_{\lambda_{1}^j}\otimes \ldots \otimes v_{\lambda_{n}^j}) \in \mathfrak{F}_{n - 1}$. 
So $x \in \mathfrak{S} + B(M \otimes_R G \otimes_R V) + Z(\mathfrak{F}_{n - 1})$. This shows that $Z(\mathfrak{F}_{n}) \subset \mathfrak{S} + B(M \otimes_R G \otimes_R V) + Z(\mathfrak{F}_{n - 1})$. But $Z(\mathfrak{F}_{1}) \subset \mathfrak{S}$. So  $Z(\mathfrak{F}_{n}) \subset \mathfrak{S} + B(M \otimes_R G \otimes_R V)$ by induction on $n$. So $Z(M \otimes_R G \otimes_R V) = \cup_{n \geq 1} Z(\mathfrak{F}_{n}) \subset  \mathfrak{S} + B(M \otimes_R G \otimes_R V)$.  Therefore, $M$ is a Golod module.
\end{proof}

\begin{remark}
The above proposition is a generalisation  of Lemma \ref{p6.2} for Golod modules. The equivalence of $(i)$ and $(ii)$ is implicit in \cite[Theorem 1.1, (iv)]{lev1}. It follows from the fact that the Eilenberg-Moore spectral sequences (see \cite[Proposition 3.2.4]{av}) 
$$E^2_{p, q} = {\Tor}_p^{\Ho(R)}(\Ho(M), k)_q \implies {\Tor}^R_{p + q}(M, k)$$
degenerates at $E^1$ page if Massey products vanish on Koszul homologies. If a basis of Koszul homologies of a local ring admits trivial Massey operations, then an explicit construction of the minimal free resolution of its residue field is available \cite[Theorem 5.2.2]{av}. 
We do not know analogous construction of the minimal free resolution of Golod modules.
\end{remark}

\begin{corollary}\label{p7} 
If $M$ is a Golod module, then the multiplication $\Ho_i(R) \otimes_R \Ho_j(M) \rightarrow \Ho_{i + j}(M)$ is zero for $i \geq 1$.
\end{corollary}

The proof of the following result is inspired by \cite[Lemma 1.6]{lev3}. 
 
\begin{corollary}\label{p7.1} 
Let $M$ be an $R$-module. Let $X \subset Z_{\geq 1}(K^R)$ and $Y \subset Z(K^M)$ be such that $X^2 = 0$, $XY = 0$, $Z_{\geq 1}(K^R ) \subset X + B(K^R)$ and $Z(K^M ) \subset Y +  B(K^M)$. Then $R$ is a Golod ring and $M$ is a Golod $R$-module.
\end{corollary}

\begin{proof}
We select a basis $\mathfrak{b}_R = \{h _{\lambda} : h_{\lambda} \in X \}_{\lambda \in \Lambda}$ of $\Ho_{\geq 1}(R)$, and a basis $\mathfrak{b}_M = \{m_{\lambda} : m_{\lambda} \in Y \}_{\lambda \in \Lambda'}$ of $\Ho(M)$. A trivial Massey operation can then be defined by setting $\mu(h_{\lambda_1}, \ldots ,h_{\lambda_p}) = 0, p \geq 2$ and $\mu(m_{\lambda_1}, h_{\lambda_2}, \ldots ,h_{\lambda_p})  = 0, p \geq 2$.
\end{proof}

If $I$ is a quasi-homogeneous ideal in a polynomial algebra $k[X_1, \ldots, X_n]$, we denote by $\partial(I)$ the ideal generated by $\frac{\partial(g)}{\partial(X_j)}$, $g \in I$ and $j = 1, 2, \ldots, n$. By Euler's theorem for quasi-homogeneous polynomials $I \subset \partial(I)$ if $\Ch(k) = 0$.

We are now equipped to prove the main result of this section extending  \cite[Theorem 1.1]{herzog1}.

\begin{theorem}\label{huneke}
Let $S = k[X_1, \ldots, X_n], \deg (X_i) = a_i > 0$ be a graded polynomial ring over a field $k$ of characteristic zero and $M$ be a graded $S$-module. Let $I$ be a proper homogeneous ideal of $S$ such that $\partial(I)^2 \subset I$ and $\partial I \subset \ann M$. Then $M$ is a Golod $S/I$-module.
\end{theorem}

\begin{proof}
We choose a set of cycles $X \subset K_{\geq 1}^{S/I}$ whose homology classes form a $k$-basis of $\Ho_{\geq 1}(S/I)$ as given by Theorem \ref{pher}. The cycles in $X$  are linear combination of generators of $\partial(I)$ modulo $I$. If $\partial(I)^2 \subset I \subset \partial I \subset \ann M$, we have $X^2 = 0$ and $X K^M = 0$ since $\partial(I)M = 0$. Therefore, the result follows from Corollary \ref{p7.1}. 
\end{proof}

\begin{corollary}\label{p7.2}
Let $S = k[X_1, \ldots, X_n], \deg (X_i) = a_i > 0$ be a graded polynomial ring over a field $k$ of characteristic zero and $M$ be a graded $S$-module. Let $I$ be a homogeneous ideal of $S$. Then $M/I^iM$ is a Golod module over $S/I^j$ for $i < j$.

\end{corollary}

\section{Algebra retracts}
We begin with the following lemma which shows that the Koszul complex of an algebra retrarct is a retract of DG algebras. 

\begin{lemma}\label{p2}
Let  $j : (R, \mathfrak{m}) \rightarrow (A, \mathfrak{n})$ be an algebra retract with a section $p: (A, \mathfrak{n}) \rightarrow (R, \mathfrak{m})$. Then the induced map on Koszul complexes $j' : K^R \rightarrow K^A$ is a retract of DG algebras. 
\end{lemma}

\begin{proof}
It is easy to see that $A = R \oplus M$ for $M =  \ker(p)$, $\mathfrak{n} = \mathfrak{m} \oplus M$, $j(r) = (r, 0), p(r, m) = r$. The product is defined by 
$$(r_1, m_1)(r_2, m_2) = (r_1r_2, r_1m_2 + r_2m_1 + m_1m_2).$$

Note that $\mathfrak{n}/\mathfrak{n}^2 = \mathfrak{m}/\mathfrak{m}^2 \oplus M/\mathfrak{n}M$. Let $\mathfrak{m}$ be minimally generated by $\{g_1, g_2, \ldots, g_m\}$. Then this set can be extended to a minimal generating set of $\mathfrak{n}$ viz. $\{g_1, g_2, \ldots, g_m, h_1, h_2, \ldots, h_n \}$ for  $h_j \in M$. We have $M = \mathfrak{n}M + Ah_1 + Ah_2+ \ldots + Ah_n$. By Nakayama's lemma  we have $M = Ah_1 + Ah_2+ \ldots + Ah_n$.

Let $K^R = R\langle e_i | \partial(e_i) = g_i \rangle$, $K^A = A\langle e_i, f_j | \partial(e_i) = g_i, \partial(f_j) = h_j \rangle$ be  Koszul complexes of $R$ and $A$ respectively. The map $j$ will induce a DG $R$-algebra homomorphism $j' : K^R \rightarrow K^A$ sending $e_i$ to $e_i$. Note that the ideal $(f_j, h_j)$ is a DG ideal of $K^A$ and $K^A/(f_j, h_j) \cong K^R$. Therefore, we have a quotient map $p' : K^A \rightarrow K^R$. It is easy to see that the composition $K^R \xrightarrow{j'} K^A  \xrightarrow{p'} K^R$ is the identity. 
\end{proof} 

The following result shows that the Golod property descends along an algebra retract. 
\begin{theorem}\label{p6}
Let $j : (R, \mathfrak{m}) \rightarrow (A, \mathfrak{n})$ be an algebra retract with a section $p: (A, \mathfrak{n}) \rightarrow (R, \mathfrak{m})$. Let $M$ be a finitely generated $R$-module which is Golod when viewed as an $A$-module via the homomorphism $p$. Then $M$ is also a Golod $R$-module.
\end{theorem}

\begin{proof}
Let $F_{*} \rightarrow R$ be a minimal $A$-free resolution of $R$ and $G_{*} \rightarrow M$ be a minimal $R$-free resolution of $M$. Then $H_* = F_{*} \otimes_R G_{*}  \rightarrow M$ is a minimal $A$-free resolution of $M$ by Theorem \ref{p5}. So we have morphisms of complexes $f : G_* \rightarrow H_*$ and $g : H_* \rightarrow G_*$ such that the composition $g \circ f = id$. By Lemma \ref{p2} we have maps on Koszul complexes viz. $j' : K^R \rightarrow K^A$ and $p' : K^A \rightarrow K^R$ such that the composition $p' \circ j' = id$. So the composition of maps of double complexes $G_* \otimes_R K^R  \xrightarrow{f \otimes j'} H_* \otimes_A K^A \xrightarrow{g \otimes p'} G_* \otimes_R K^R$ is the identity.

Let $\{{}^RE^r_{pq}, {}^Rd^r_{pq}\}$, $\{{}^AE^r_{pq}, {}^Ad^r_{pq}\}$ be homology spectral sequences corresponding to the double complexes $G_* \otimes_R K^R, H_* \otimes_A K^A$ respectively. Then we have morphisms of spectral sequences $\alpha^r_{pq} : {}^RE^r_{pq} \rightarrow {}^AE^r_{pq}$ and $\beta^r_{pq} : {}^AE^r_{pq} \rightarrow  {}^RE^r_{pq}$ such that the composition $\beta^r_{pq} \circ \alpha^r_{pq} = id$. 

We have ${}^AE^2_{pq} = \Tor^A_p(M, \Ho_q(A)) \underset{p}{\Rightarrow} \Ho_{p + q}(A;M)$. Since $M$ is a Golod $A$-module, $\{{}^AE^r_{pq}, {}^Ad^r_{pq}\}$ satisfies conditions of Theorem \ref{p4} by Remark \ref{p4s}. So $\{{}^RE^r_{pq}, {}^Rd^r_{pq}\}$ will also satisfy the same. Hence $M$ is also a Golod $R$-module by Remark \ref{p4s}.
\end{proof}

The method of proof of the following lemma is inspired by the work of Herzog in \cite{her1}.
\begin{lemma}\label{m1}
Let $j:(R, \m) \rightarrow (A, \n)$ be an algebra retract which admits sections (possibly equal) $p_1, p_2$. Let $\ker(p_1) = I_1$ and $\ker(p_2) = I_2$ satisfy $I_1I_2 = 0$. Let $N$ be an $R$-module. If we consider $N$ as an $A$-module via any of the maps $p_i$'s, then $N$  is a $j$-Golod module and 
\[P^A_N(t) = \frac{P^R_N(t)}{1 - t(P^R_A(t) - 1)} = \frac{P^R_N(t)}{1 - tP^R_{I_i}(t)} \ ,i = 1, 2.\] 
\end{lemma}

\begin{proof}
The composition $R \xrightarrow{j} A \xrightarrow{p_{i}} R$ is the identity. So we have a decomposition $A = R \oplus I_i$  as an $R$-module for $i = 1,2$. Since $I_i$'s are finite $A$-modules and $I_1I_2 = 0$, $I_i$'s are so as $R$-modules. We have $P^R_A(t) = 1 + P^R_{I_i}(t)$ for $i = 1, 2$. So $P^R_{I_1}(t) = P^R_{I_2}(t) = P$ for some $P \in \ZZ[|t|]$. We also have $\ITor^R(A, k) = \Tor^R(I_1, k) = \Tor^R(I_2, k)$. Since 
$I_1I_2 = 0$, we have $\n \ITor^R(A, k) = 0$.

Let $R_i$ denote the $A$-module whose underlying set is $R$ and the multiplication be defined by $a\cdot r = p_i (a) r$, $a \in A, r \in R_i = R$. Without loss of generality, we assume $N$ as an $A$-module via the map $p_1$. By Theorem \ref{p5} we have $P^A_N(t) = P^A_{R_1}(t) P^R_N(t)$. So it is enough to show that $P^A_{R_1}(t) = \frac{1}{1 - tP}$.

We have the following short exact sequence of $A$-modules.
\[ 0 \rightarrow I_1 \rightarrow A \xrightarrow{p_1} R_1 \rightarrow 0.\]
So we have $P^A_{R_1}(t) = 1 + tP^A_{I_1}(t)$. Since $A = R \oplus I_{2}$ and $I_1I_2 = 0$, the $A$-module structure on $I_1$ is coming from an $R$-module structure of $I_1$ via the map $p_2$. So $P^A_{I_1}(t) = P^A_{R_2}(t)P^R_{I_1}(t)$ and 
$P^A_{R_1}(t) = 1 + tP^R_{I_1}(t)P^A_{R_2}(t) = 1 + tPP^A_{R_2}(t)$.

If $p_1 = p_2$, we have $R_1 = R_2 = R$ as $A$-module. So we have $P^A_{R_1}(t) = 1 + tPP^A_{R_1}(t)$ and $P^A_{R_1}(t) = \frac{1}{1 - tP}$.

Otherwise by a similar argument we have $P^A_{R_2}(t) = 1 + tPP^A_{R_1}(t)$. Now
\begin{align*}
P^A_{R_1}(t)
& =  1 + tPP^A_{R_{2}}(t)\\
& = 1 + tP + t^2P^2 P^A_{R_{1}}(t), \ \text{since}\; P^A_{R_2}(t) = 1 + tPP^A_{R_{1}}(t)
\end{align*}
Therefore, we have 
$$P^A_{R_1}(t) = \frac{1 + tP}{1 - t^2P^{2}} = \frac{1}{1 - tP}.$$
\end{proof}

\begin{lemma}\label{m2.5}
Let $j:(R, \m) \rightarrow (A, \n)$ be an algebra retract which admits sections (possibly equal) $p, p'$. Let $\ker(p) = I$ and $\ker(p') = I'$ satisfy $II' = 0$. Consider $R$-module structures on $I$, $I'$ via the map $j$. Then we have

\begin{enumerate}
\item Both the ideals $I$ and $I'$ have the same Koszul polynomials i.e. $\kappa_I(t) = \kappa_{I'}(t)$.

\item $P^R_I(t) = P^R_{I'}(t)$.

\item $\mu_A(\n) = \mu_R(\m) + \mu_R(I) =  \mu_R(\m) + \mu_R(I')$. 
Here $\mu(-)$ denote the minimal number of generators.

\end{enumerate}
\end{lemma}

\begin{proof}
We have the decomposition $A = R \oplus I = R \oplus I'$ as $R$-modules. Since $I_i$'s are finite $A$-modules and $I_1I_2 = 0$, $I_i$'s are so as $R$-modules. We have $K^R \oplus K^I = K^R \oplus K^{I'}$ and $\Ho_i(R) \oplus \Ho_i(I) = \Ho_i(R) \oplus \Ho_i(I')$ for all $i$. This gives $\dim_k \Ho_i(I) = \dim_k \Ho_i(I')$ for all $i$ and therefore $\kappa_I^R(t) = \kappa_{I'}^R(t)$. This proves (1).

We have $P^R_A(t) = 1 + P^R_I(t) = 1 + P^R_{I'}(t)$. So $P^R_I(t) = P^R_{I'}(t)$. This proves (2).

Now we prove (3). We have $\n = \m \oplus I = \m \oplus I'$. Since $II' = 0$, we have $I^2 \subset \n I \subset \m I$ and $\n^2 = \m^2 \oplus (\m I + I^2) = \m^2 \oplus \m I$. Therefore, $\n/ \n^2 = \m/\m^2 \oplus I/ \m I$ as $k = R/ \m$ vector space. So $\mu_A(\n) = \mu_R(\m) + \mu_R(I)$. The other equality for $I'$ follows similarly. 
\end{proof}

Now we are in a position to prove the main result of this section.
\begin{theorem}\label{m2}
Let $j:(R, \m) \rightarrow (A, \n)$ be an algebra retract which admits sections (possibly equal) $p$ and $p'$. Let $\ker(p) = I$ and $\ker(p') = I'$ satisfy $II' = 0$. 
Consider $R$-module structures on $I$, $I'$ via the retract map $j$.
Then the ideal $I$ is a Golod $R$-module if and only if $I'$ is so. 

Let $N$ be an $R$-module. If we consider $N$ as an $A$-module via any of the maps $p, p'$, then $N$ is a Golod $A$-module if and only if $N$ is a Golod $R$-module and $I$ (equivalently $I'$) is a Golod $R$-module. In particular, $A$ is a Golod ring if and only if $I$ (equivalently $I'$) is a Golod $R$-module.
\end{theorem}

\begin{proof} 
The first part follows from $(1)$ and $(2)$ of Lemma \ref{m2.5}. Let $\mu_R(I) = \mu_R(I') = n$.

\noindent
{\bf Claim 1:} Let $\{g_1, g_2, \ldots, g_m\}$ be a minimal generating set of $\m$ and $\{h_1, h_2, \ldots, h_n\}$ be that of $I$. We claim that we can choose a minimal generating set $\{h_1', h_2', \ldots, h_n'\}$ of $I'$ such that $h_i - h_i' \in \m A$.

\noindent
{\bf Proof of the claim:} Let $\{h_1', h_2', \ldots, h_n'\}$ be any minimal generating set of $I'$. Then both  $\{g_1, g_2, \ldots, g_m, h_1, h_2, \ldots, h_n\}$ and $\{g_1, g_2, \ldots, g_m, h_1', h_2', \ldots, h_n'\}$ are minimal generating sets of $\n$. So we have a matrix of the form 
$B = \begin{pmatrix}
I_m & *\\ 
0 & C
\end{pmatrix} \in \M_{m + n}(A), C \in \M_{ n}(A)$ such that $(g_1, \ldots, g_m, h_1,\ldots, h_n) = (g_1, \ldots, g_m, h_1',\ldots, h_n')B$.
Now $B$ is nonsingular modulo $\n$ since its image $\bar{B} \in \M_n(k)$, $k = R/ \m = A/ \n$ is the matrix of change of coordinates with respect to two different bases of $\n/\n^2$. So $B \in \GL_{m + n}(A)$. This gives $C \in \GL_n(A)$. The minimal generating set $\{h_1'', h_2'', \ldots, h_n''\}$ of $I'$ defined by $(h_1'', h_2'', \ldots, h_n'') = (h_1', h_2', \ldots, h_n')C$ will do the job.

Now we prove the last part of the theorem. We have  the following equality by Lemma \ref{m1}. 
\begin{equation}\label{eq1}
P^A_N(t) = \frac{P^R_N(t)}{1 - tP^R_{I}(t)}.
\end{equation}

We assume that both $N$ and $I$ are  Golod $R$-modules. Then we have
$P^R_N(t) = \frac{\kappa_N^R(t)}{1 - t(\kappa_R(t) - 1)}$ and $P^R_I(t) = \frac{\kappa_I^R(t)}{1 - t(\kappa_R(t) - 1)}$  (see Definition \ref{p3}). Therefore, 

\begin{equation}\label{eq2.5}
\frac{P^R_I(t)}{P^R_N(t)} = \frac{\kappa_I^R(t)}{\kappa_N^R(t)}
\end{equation}

We choose minimal sets of generators $\{g_1, g_2, \ldots, g_m\}$ of $\m$, $\{h_1, h_2, \ldots, h_n\}$ of $I$ and $\{h_1', h_2',\\ \ldots, h_n'\}$ of $I'$ obtained by the claim 1. Let $K^A = \langle X_l, Y_s | \partial(X_l) = g_l, \partial(Y_s) = h_s' \rangle$ be the Koszul complex of $A$ on the minimal set of generators $\{g_1, g_2, \ldots, g_m, h_1', h_2', \ldots, h_n'\}$ of $\n$. We have the following short exact sequence of $A$-modules.
\[0 \rightarrow I \rightarrow A \xrightarrow{p} R \rightarrow 0.\]
This gives the following short exact sequence of complexes.
\[0 \rightarrow K^A \otimes_A I \rightarrow K^A \xrightarrow{p} K^A \otimes_A R \rightarrow 0.\] 

Now $II' = 0$. So $ K^A \otimes_A I = K^I \otimes_R \Lambda(\bigoplus_{s= 1}^n RY_s)$ and $\Ho(K^A \otimes_A I) = \Ho(I) \otimes_k \Lambda (\bigoplus_{s = 1}^n kY_s)$. Since $h_s - h_s' \in \m A$, we can choose linear polynomials $f_1, f_2, \ldots, f_n \in A\langle X_1, X_2, \ldots, X_m \rangle \subset K^A$ such that $\partial(Y_s + f_s) = h_s' + f_s(g_1, \ldots, g_m) = h_s$. Let $\overline{f_s} \in R\langle X_1, X_2, \ldots, X_m\rangle \subset K^R$ denote the image of $f$ under $p$ and $Z_s = Y_s + \overline{f_s} \in K^A \otimes_A R$. Note that $\partial(Z_s) = 0$ in  $K^A \otimes_A R$. Therefore, we have $K^A \otimes_A R = K^R \otimes_R \Lambda(\bigoplus_{s= 1}^n RZ_s)$ and $\Ho(K^A \otimes_A R) = \Ho(R) \otimes_k \Lambda (\bigoplus_{s = 1}^n kZ_s)$.

Without loss of generality, we assume $N$ as an $A$-module via the map $p$. By a similar argument for the $R$-module $N$ we have  $K^A \otimes_A N = K^N \otimes_R \Lambda(\bigoplus_{s= 1}^n RZ_s)$ and $\Ho(K^A \otimes_A N) = \Ho(N) \otimes_k \Lambda (\bigoplus_{s = 1}^n kZ_s)$. Therefore, 
\begin{equation}\label{eq3}
\kappa_N^A(t) = \kappa_N^R(t)(1 + t)^n
\end{equation}

We study the boundary map $\delta : H (K^A \otimes_A R) \rightarrow \Ho(K^A \otimes_A I)[-1]$. The boundary map $\delta$ is $\Ho(R)$ linear and $\delta(Z_{s_1} \wedge Z_{s_2} \ldots \wedge Z_{s_{u}}) = \sum c_{l_1 \ldots l_v} Y_{l_1} \wedge \ldots Y_{l_v} $ for $c_{l_1 \ldots l_v} \in \Ho_{u - v - 1}(I)$.

%
\vskip 0.5cm
\noindent
{\bf Claim 2}:
The restriction $\delta_{res} : \Lambda_{r + 1}(\bigoplus_{s =1}^{n} kZ_s) \rightarrow \Ho_r(K^A \otimes_A I)$ is an injection for $r \geq 0$.

We have  $\Ho_r(K^A \otimes_A I) = \bigoplus_{u + v = r}\Ho_u(I) \otimes_k \Lambda_v (\bigoplus_{s =1}^n kY_s)$. Let $pr : \Ho_r(K^A \otimes_A I) \rightarrow I/\m I \otimes_k \Lambda_r (\bigoplus_{s =1}^n kY_s)$ denote the projection. The composition $pr \circ \delta_{res}$ is given by $Z_{s_1} \wedge Z_{s_2} \ldots \wedge Z_{s_{r + 1}} \mapsto \sum_{l = 1}^{r + 1} (-1)^{ l + 1} \bar{h}_{s_l} Y_{s_1} \wedge Y_{s_{l - 1}} \wedge Y_{s_{l + 1}} \wedge \ldots \wedge Y_{s_{r + 1}}$. This is an injection as $\bar{h}_1, \ldots \bar{h}_n$ is a vector space basis of $I/\m I$. So $\delta_{res}$ is also an injection.
Therefore, we have 
\begin{equation}\label{eq4}
\kappa_A(t) 
\leq \kappa_I^R(t)( 1 + t)^n + \kappa_R(t)(1 + t)^n - (1 + t)\frac{(1 + t)^n - 1}{t}
\end{equation}

Note that $I$ is a Golod $R$-module by our assumption. So by Corollary \ref{p7} the boundary map $\delta$ is zero on $\bigoplus_{r > 0}\Ho_{r}(R) \otimes_k \Lambda(\bigoplus_{s =1}^n kZ_s)$. This is so because
\begin{align*}
&\delta(z \otimes Z_{s_1} \wedge Z_{s_2} \ldots Z_{s_{u}})\\
& =  (-1)^{|z|} z \otimes_k \delta(Z_{s_1} \wedge Z_{s_2} \ldots Z_{s_{u}})\\
&=(-1)^{|z|} \sum z \cdot c_{l_1 \ldots l_v} Y_{l_1} \wedge \ldots Y_{l_v}\, \text{for}\, z \in {\Ho}_r(R)\, \text{and}\, c_{l_1 \ldots l_v} \in {\Ho}_{u - v - 1}(I).
\end{align*}

So equality holds in \ref{eq4}. We have $\kappa_A(t) = \kappa_I^R(t)( 1 + t)^n + \kappa_R(1 + t)^n - (1 + t)\frac{(1 + t)^n - 1}{t}$. This implies

\begin{equation}\label{eq5}
1 - t(\kappa_A(t)  - 1) = ( 1 + t)^n\{(1 + t) - t(\kappa_I^R(t) + \kappa_R(t)) \}
\end{equation}

Now we have
\begin{align*}
\frac{1}{P^A_N(t)} & = \frac{1}{P^R_N(t)} - t\frac{P^R_I(t)}{P^R_N(t)} \quad \text{by equation \ref{eq1}}\\
&= \frac{1 - t(\kappa_R(t) -1)}{\kappa_N^R(t)} - t \frac{\kappa_I^R(t)}{\kappa_N^R(t)} \quad \text{by equation \ref{eq2.5} }\\
&= \frac{(1 + t) - t(\kappa_R(t) + \kappa_I^R(t))}{\kappa_N^R(t)}\\
&= \frac{(1 + t)^n\{(1 + t) - t(\kappa_R(t) + \kappa_I^R(t))\}}{\kappa_N^A(t)} \quad \text{by equation \ref{eq3}}\\
&= \frac{1 - t(\kappa_A(t)  - 1)}{\kappa_N^A(t)}\quad \text{by equation \ref{eq5}}\\
\end{align*}
Hence $N$ is a Golod $A$-module whenever both $N$, $I$ are Golod $R$-modules.

Now we prove the converse. We assume that $N$ is a Golod $A$-module. Then by Theorem \ref{p6}, $N$ is a Golod $R$-module. Note that $A$, $R$ are Golod rings by Theorem \ref{levingm}. So both $K^R$, $K^A$ admit trivial Massey products on homologies. 

We recall $K^R = \langle X_l | \partial (X_l) = g_l \rangle$, $K^I = K^R \otimes I$ and $K^A = \langle X_l, Y_s | \partial(X_l) = g_l, \partial(Y_s) = h_s' \rangle$. We write $\Ho(I) = \Ho(K^I)$, $\Ho(A) = \Ho(K^A)$.
It is easy to see that the map $\Ho(I) \rightarrow \Ho(A)[n]$ defined by $z \mapsto Y_1 \wedge Y_2 \ldots \wedge Y_n \wedge  z$ is an injective $\Ho(R)$-module homomorphism.

Let  $v_1 \in \Ho(I)$ and $v_2, v_3, \ldots, v_{n} \in \Ho_{\geq 1}(R)$ . Let  $\{a_{ij} | 0 \leq i < j \leq n, (i\ j) \not= (0 \ n) \}$ be a defining system for $v_1, v_2, \ldots, v_{n - 1}, v_n$. Suppose $\wedge Y_{\ast}$ denotes $Y_1 \wedge Y_2 \ldots \wedge Y_n$.
We define 
\begin{align*}
b_{ij} = 
\begin{cases}
(\wedge Y_{\ast}) \wedge a_{ij}  \ & \text{if} \ i = 0\\
a_{ij} \ &\text{if} \ 0 < i < j \leq n.
\end{cases}
\end{align*}

Then it is easy to see that $\{b_{ij} | 0 \leq i < j \leq n, (i\ j) \not= (0 \ n) \}$ is a defining system for $(\wedge Y_{\ast}) \wedge v_1, v_2, \ldots, v_n $. Since $A$ is a Golod ring the element $\sum_{k = 1}^{n - 1} \overline{b_{0 k}}b_{k n} =  - \overline{\wedge Y_{\ast}}(\sum_{k = 1}^{n - 1} \overline{a_{0 k}}a_{k n})$ is a boundary in $K^A$. So $\sum_{k = 1}^{n - 1} \overline{a_{0 k}}a_{k n}$ is a boundary in $K^I$. Therefore, $\langle v_1, \ldots, v_n \rangle = 0$ and $I$ is a Golod $R$-module by Proposition \ref{p6.3} $(ii) \implies (i)$.
\end{proof}

\section{applications}
In this section we apply our main result to find new constructions of Golod rings. The following theorem gives us a method to study the Golod property of modules using results available to characterise Golod rings.

\begin{theorem}\label{ap1}
Let $(R, \m)$ be a local ring and $M$ be an $R$-module. Let $A = R \lJoin M$ be the trivial extension of $R$ by $M$. Then $A$ is a Golod ring if and only if $M$ is a Golod $R$-module. 
\end{theorem}

\begin{proof}
We have $j : R \rightarrow A,\ j(r) = (r, 0)$ and $p : A \rightarrow R,\ p(r, m) = r$ such that $p \circ j = id$. We have $\ker(p)^2 = M^2 = 0$. 
The proof now follows from Theorem \ref{m2}.
\end{proof}

\begin{remark}
The above result gives us a new class of Golod ideals in a polynomial algebra. For example if $\m = (X_1, \ldots, X_n)$ in $R = k[X_1, \ldots, X_n]$, $\n = (Y_1, \ldots, Y_m)$ in $S = R[Y_1, \ldots, Y_m]$ and $\Ch(k) = 0$, then the ideal $I = \m^s + \m^t\n + \n^2, s > t$ is a Golod ideal. This is because $S/I$ is a trivial extension of $R/\m^s$ by the Golod module $\oplus_{i = 1}^m \frac{R}{\m^t}e_i$ (see Corollary \ref{p7.2}). Note that $\partial(I)^2 \not\subset I$ if $s > 2t$. If $s \leq t$, then $I$ is not a Golod ideal as the ring $\frac{k[X_1, Y_1]}{(X_1^s, Y_1^2)}$ is a retract of $S/I$ and $\frac{k[X_1, Y_1]}{(X_1^s, Y_1^2)}$ is not a Golod ring. We do not know a complete classification of $(s, u, t, v) \in \NN^4$ such that $\m^s + \m^t\n^u + \n^v$ is a Golod ideal in $S$.
\end{remark}

Our next result characterises regular local rings in terms of the Golod property of canonical modules.

\begin{corollary}\label{ap4}
Let $(R, \m)$ be a Cohen-Macaulay local ring which admits the canonical module $\omega_R$. Then the following are equivalent.
\begin{enumerate}
\item[(1)]
$R$ is a regular local ring.

\item[(2)]
$\omega_R$ is a Golod $R$-module.
\end{enumerate}

\end{corollary}

\begin{proof}
We only need to show $(2) \implies (1)$. Let $A$ be the trivial extension of $R$ by $\omega_R$. Then $A$ is Golod ring and $\Ho_{\geq 1}(A)^2 = 0$. Since $A$ is also a Gorenstein ring, $\Ho(A)$ is a Poincar\'e algebra (cf. \cite{av0}). So $co\depth(A) \leq 1$. Note that $\depth(R) = \depth(A)$ and $emb\dim(A) \geq 1 + emb\dim(R)$. So $co\depth(R) = 0$. Therefore, $R$ is a regular local ring.
\end{proof}


%

\begin{definition}{\rm (Fibre product)}\label{fibre}
Let $\e_i : (R_i, \m_i) \twoheadrightarrow (S, \n), i = 1, 2$ be local epimorphisms. We define the fibre product of $\e_1, \e_2$ as the ring $R_1 \times_S R_2 = \{(r_1, r_2) : \e_1(r_1) = \e_2(r_2)\}$. When the surjections  $\e_1, \e_2$ are clear from the context, we say that the ring $R_1 \times_S R_2$ is the fibre prduct of $R_1$ and $R_2$ over $S$.
\end{definition}

\begin{lemma}\label{ap2}
Let $I$ be an ideal of a local ring $R$. Let $A  = R\times_{R/I}R$. Then the following are equivalent.
\begin{enumerate}

\item The fibre product $A$ is a Golod ring. 

\item The ideal $I$ is a Golod $R$-module.

\item
The ideal $I \oplus I$ of $A$ is a Golod $A$-module.
\end{enumerate}
\end{lemma}

\begin{proof}
We have $j : R \rightarrow A,\ j(r) = (r, r)$, $p : A \rightarrow R, \ p(r, r') = r$ and $p' : A \rightarrow R,\ p'(r, r') = r'$. We have $p \circ j = p' \circ j = id$ and $\ker(p) \ker{p'} = 0$. The equivalence of (1) and (2) clearly follows from Theorem \ref{m2}. 

If $I$ is a Golod $R$-module, then $I$ is a Golod $A$-module via any of the maps $p, p'$ by Theorem \ref{m2}. So the direct summand $I \oplus I$ is a Golod $A$-module. This proves that (2) implies (3). The proof of (3) implying (1) follows from Theorem \ref{levingm}. 
\end{proof}

Now we generalize the above result as follows:

\begin{theorem}\label{ap3}
Let $I$ be an ideal of a local ring $R$. Let for $n \geq 2$

\[A_n = \underbrace{R \times_{R/I} R \ldots \times_{R/I} R}_{n\ \text{times}}.\]
Then the following are equivalent.
\begin{enumerate}
\item $I$ is a Golod $R$-module.

\item The fibre product $A_n$ is a Golod ring for all $n \geq 2$.

\item The fibre product $A_n$ is a Golod ring for some $n \geq 2$.
\end{enumerate}
\end{theorem}

\begin{proof}
We have $A_n = \{ (r_1, \ldots, r_n) \in R^n : r_1 \equiv  \ldots \equiv r_n \vpmod I\}$. The ring homomorphism $\phi : A_m \rightarrow A_n$, $m < n$ defined by $\phi(r_1, \ldots, r_m) = (r_1, \ldots, r_n)$, $r_{m + i} = r_m$ for $i = 1, \ldots, n - m$ is an algebra retract.

Let $I$ be a Golod $R$-module. To prove $(2)$ it is enough to prove that $A_{2^n}$ is a Golod ring for all $n \geq 1$ since each $A_m$ is an algebra retract of $A_{2^n}$ for sufficiently large $n$.  By Theorem \ref{levingm} it suffices to prove that the ideal $J_{2^n} = \underbrace{I \oplus I \ldots \oplus I}_{2^n\ \text{times}}$ of $A_{2^n}$ is a Golod $A_{2^n}$-module for all $n \geq 1$. The statement is clear for $n = 1$ by Lemma \ref{ap2}. Now we assume that $J_{2^{n}}$ is a Golod $A_{2^{n}}$-module. Note that $A_{2^{n+1}} = A_{2^{n}} \times_{A_{2^{n}} / J_{2^{n}}} A_{2^{n}}$. So the ideal $J_{2^{n+1}}$ is a Golod $A_{2^{n+1}}$-module by Lemma \ref{ap2}. Therefore the statement is true by induction. Hence $(1)\implies(2)$ is proved.

$(2)\implies(3)$ is obvious.

The ring $A_2$ is an algebra retract of $A_n$ for $n \geq 2$. If $A_n$ is a Golod ring for some $n$, then  so is $A_2$. Then $I$ is a Golod $R$-module by Lemma \ref{ap2}. This establishes $(3)\implies(1)$.
\end{proof}

\begin{lemma}\label{ap6}
Let $f : (R, \m, k) \twoheadrightarrow (S, \n, k)$ be a large homomorphism of local rings and $M$ be an $S$-module. Let $R \lJoin M, S \lJoin M$ denote trivial extensions of $R$, $S$ by $M$ respectively. Then the induced map $g :R \lJoin M \twoheadrightarrow S \lJoin M$ defined by $g(r, m) = (f(r), m)$ is also a large homomorphism.
\end{lemma}

\begin{proof} 
Any $S$-module $N$ has an $R$-module structure via $f$. So $N$ also has $R \lJoin M$ and $S \lJoin M$-module structures by defining the action of $M$ on $N$ to be trivial. We prove by induction on $n$ that the induced map $ \Tor^{R \lJoin M}_n(N, k) \rightarrow \Tor^{S \lJoin M}_n(N, k)$ is a surjection  for any $S$-module $N$.

Let $P_* \twoheadrightarrow N$ be a free resolution of $N$ as $S$-module and $Q_* \twoheadrightarrow S$ be that of $S$ as $S \lJoin M$-module. Then by Theorem \ref{p5}, the complex $Q_* \otimes_S P_* \twoheadrightarrow N$ is a free resolution of $N$ as $S \lJoin M$-module. Therefore, 
\begin{align*}
{\Tor}^{S \lJoin M}(N, k) &= \Ho((Q_* \otimes_S P_*) \otimes_{S \lJoin M} k) \\
&= \Ho((Q_* \otimes_{S \lJoin M} k ) \otimes_k (P_* \otimes_S k))\\
&= {\Tor} ^{S \lJoin M}(S, k) \otimes_k {\Tor}^S(N, k)
\end{align*}

Similarly, ${\Tor}^{R \lJoin M}(N, k) = {\Tor} ^{R \lJoin M}(R, k) \otimes_k {\Tor}^R(N, k)$. The induced map ${\Tor}^R(N, k) \rightarrow {\Tor}^S(N, k)$ is surjective by Theorem \ref{large}. We have the following commutative diagram.
\[
\xymatrix{
0 \ar[r]& M \ar@{=}[d] \ar[r] & R \lJoin M \ar[d]^g \ar[r] & R \ar[d]^f \ar[r]& 0\\
0 \ar[r]& M\ar[r]&  S \lJoin M \ar[r]&  S \ar[r]& 0
}
\]

The long exact sequence of Tor gives the following.
\[
\xymatrix{
{\Tor}_i ^{R \lJoin M}(R, k) \ar[r]^{\cong} \ar[d]& {\Tor}_{i - 1}^{R \lJoin M}(M, k)\ar[d]&\\
{\Tor}_i ^{S \lJoin M}(S, k) \ar[r]^{\cong}& {\Tor}_{i - 1}^{S \lJoin M}(M, k)
}
\]
By induction hypothesis ${\Tor}_{i}^{R \lJoin M}(M, k) \rightarrow {\Tor}_{i}^{S \lJoin M}(M, k)$ is surjective for $i \leq n - 1$. So ${\Tor}_i ^{R \lJoin M}(R, k) \rightarrow {\Tor}_i ^{S \lJoin M}(S, k)$ is surjective for $i \leq n$. Therefore,  ${\Tor}_n^{R \lJoin M}(N, k) \rightarrow {\Tor}_n^{S \lJoin M}(N, k)$ is surjective. So the induction is complete and we have that ${\Tor}^{R \lJoin M}(N, k) \rightarrow {\Tor}^{S \lJoin M}(N, k)$ is surjective. Now the result follows by choosing $N = k$.
\end{proof}

The following is suggested by Iyengar.

\begin{lemma}\label{ap7}
Let $f : (R, \m) \twoheadrightarrow (S, \n)$ be a large homomorphism of local rings. Then $S$ is Golod if $R$ is so.
\end{lemma}

\begin{proof}
By Theorem $\ref{p9}$, it is enough to show that the homotopy Lie algebra $\pi^{\geq 2}(S)$ is free. Since $f$ is large, the induced map on Lie algebras \it{viz.} $f^* :  \pi(S) \rightarrow \pi(R)$ is an injection. The result now follows from Proposition \ref{p8}.
\end{proof}

The following theorem strengthens Theorem \ref{p6}.

\begin{theorem}\label{ap8}
Let $f : (R, \m) \twoheadrightarrow (S, \n)$ be a large homomorphism of local rings and $M$ be an $S$-module. Then $M$ is a Golod $S$-module if $M$ is so as an $R$-module.
\end{theorem}

\begin{proof}
The induced map $R \lJoin M \twoheadrightarrow S \lJoin M$ is a large homomorphism by Lemma \ref{ap6}. If $M$ is a Golod $R$-module, then $R \lJoin M$ is a Golod ring by Theorem \ref{m2}. By Lemma \ref{ap7}, the ring $S \lJoin M$ is also a Golod ring. Now the result follows by another application of Theorem \ref{m2}.
\end{proof}

We conclude this article with the following question. Note that if $R$ is a regular local ring and $I$ is an ideal of $R$, then $R \times_{R/I} R$ is always a Golod ring.

\begin{question}
Let $R_1$, $R_2$, $S$ be local rings and $\varepsilon_1 : R_1 \rightarrow S$,  $\varepsilon_2 : R_2 \rightarrow S$ be surjective local homomorphisms. Let $A = R_1 \times_S R_2$. Under which conditions is $A$ a Golod ring? If both  $R_1$ and $R_2$ are regular local rings, is $A$ a Golod ring?
 \end{question}
 
 \begin{acknowledgement}
I am indebted to Srikanth B. Iyengar for encouragement, and many inspiring discussions concerning this paper. I also owe many thanks to H. Ananthnarayan for helping me learn the subject.
 \end{acknowledgement}


\begin{thebibliography}{VAST}
\bibitem{ananth} Ananthnarayan, H.; Avramov, Luchezar L.; Moore, W. Frank; Connected sums of Gorenstein local rings. J. Reine Angew. Math. 667 (2012), 149--176. 


\bibitem{apassov}Apassov, Dmitri; Almost finite modules. Comm. Algebra 27 (1999), no. 2, 919--931. 

\bibitem{av0} Avramov, L. L.; Golod, E. S.; The homology of algebra of the Koszul complex of a local Gorenstein ring. Mat. Zametki 9 1971 53--58. 

\bibitem{av2} Avramov, Luchezar L.; Local algebra and rational homotopy. Algebraic homotopy and local algebra (Luminy, 1982), 15–43, Ast\'{e}risque, 113-114, Soc. Math. France, Paris, 1984.

\bibitem{av1}Avramov, Luchezar L.; Golod homomorphisms. Algebra, algebraic topology and their interactions (Stockholm, 1983), 59--78, Lecture Notes in Math., 1183, Springer, Berlin, 1986.
 
\bibitem{av} Avramov, Luchezar L.; Infinite free resolutions. Six lectures on commutative algebra. Bellaterra, 1996, 1--118.

\bibitem{dress}Dress, Andreas; Kr\"amer, Helmut; Bettireihen von Faserprodukten lokaler Ringe. Math. Ann. 215 (1975), 79--82. 

\bibitem{douglas}Douglas L. Costa; Retracts of polynomial rings. J. Algebra 44 (1977), 492--502.

\bibitem{epstein}Epstein, Neil; Nguyen, Hop D; Algebra retracts and Stanley-Reisner rings. J. Pure Appl. Algebra 218 (2014), no. 9, 1665--1682.

\bibitem{felix}F\'{e}lix, Yves; Halperin, Stephen; Thomas, Jean-Claude; Rational homotopy theory. Graduate Texts in Mathematics, 205. Springer-Verlag, New York, 2001.

\bibitem{gul}Gulliksen, T. H.; Levin, G.; Homology of local rings. Queen's Papers Pure Appl. Math. 20, Queen's Univ., Kingston, ON, 1969.

\bibitem{her1} Herzog, J\"urgen; Algebra retracts and Poincaré-series. Manuscripta Math. 21 (1977), no. 4, 307--314.

\bibitem{her2} Herzog, J\"urgen; Canonical Koszul cycles. Aportaciones Mat. Notas de Investigaci\'on 6 (1992) 33--41. 

\bibitem{herzog1} Herzog, J\"urgen; Huneke, Craig; Ordinary and symbolic powers are Golod. Adv. Math. 246 (2013), 89--99.

\bibitem{lem}Lemaire, Jean-Michel; Alg\`{e}bres connexes et homologie des espaces de lacets. (French) Lecture Notes in Mathematics, Vol. 422. Springer-Verlag, Berlin-New York, 1974.

\bibitem{les1}Lescot, Jack; La s\'erie de Bass d'un produit fibr\'e d'anneaux locaux. Paul Dubreil and Marie-Paule Malliavin algebra seminar, 35th year (Paris, 1982), 218--239.

\bibitem{les2} Lescot, Jack; S\'eries de Poincar\'e et modules inertes. J. Algebra 132 (1990), no. 1, 22--49. 

\bibitem{lev3} Levin, Gerson; Lectures on Golod homomorphisms. Matematiska Istitutionen, Stockholms Universitet, Preprint 15, 1975.

\bibitem{lev4}Levin, Gerson; Large homomorphisms of local rings. Math. Scand. 46 (1980), no. 2, 209--215.

\bibitem{lev1} Levin, Gerson; Modules and Golod homomorphisms. J. Pure Appl. Algebra 38 (1985), no. 2--3, 299--304.

\bibitem{may}May, J.P.; Matric Massey products. J. Algebra 12 (1969), 553--568. 

\bibitem{moore} Moore, W. Frank; Cohomology over fiber products of local rings. J. Algebra 321 (2009), no. 3, 758--773. 

\bibitem{ogoma} Ogoma, Tetsushi; Fibre products of Noetherian rings and their applications. Math. Proc. Cambridge Philos. Soc. 97 (1985), no. 2, 231--241. 

\bibitem{ogushi} Ohsugi, Hidefumi; Herzog, J\"urgen; Hibi, Takayuki; Combinatorial pure subrings. Osaka J. Math. 37 (2000), no. 3, 745--757.

\end{thebibliography}
\end{document}